\documentclass[runningheads]{llncs}
\usepackage{amsmath}
\usepackage{graphicx,subcaption,colortbl}
\usepackage{comment}
\newtheorem{assumption}{Assumption}
\usepackage{amssymb}
\usepackage{dirtytalk}
\usepackage{xcolor}
\newcommand{\ie}{{\emph{i.e., }}}
\newcommand{\eg}{{\emph{e.g., }}}
\newcommand{\ubar}[1]{\text{\b{$#1$}}}
\newcommand{\norm}[1]{\left\lVert#1\right\rVert}
\begin{document}
\title{Between Steps: Intermediate Relaxations between big-M and Convex Hull Formulations}
\titlerunning{Relaxations between big-M and convex hull formulations}
%
\author{Jan Kronqvist\thanks{Corresponding author} \and Ruth Misener \and Calvin Tsay}
\authorrunning{J. Kronqvist et al.}
%
\institute{Deparment of Computing, Imperial College London\\
\email{\{j.kronqvist, r.misener, c.tsay\}@imperial.ac.uk}\\}
\maketitle              
\begin{abstract}
This work develops a class of relaxations in between the big-M and convex hull formulations of disjunctions, drawing advantages from both.  The proposed ``$P$-split'' formulations split convex additively separable constraints into $P$ partitions and form the convex hull of the partitioned disjuncts. Parameter $P$ represents the trade-off of model size vs.\ relaxation strength. We examine the novel formulations and prove that, under certain assumptions, the relaxations form a hierarchy starting from a big-M equivalent and converging to the convex hull. 
We computationally compare the proposed formulations to big-M and convex hull formulations on a test set including: K-means clustering, P\_ball problems, and ReLU neural networks.  The computational results show that the intermediate $P$-split formulations can form strong outer approximations of the convex hull with fewer variables and constraints than the extended convex hull formulations, giving significant computational advantages over both the big-M and convex hull.

\end{abstract}

\section{Introduction}
There are well-known trade-offs between the big-M and convex hull relaxations of disjunctions in terms of problem size and relaxation tightness. Convex hull formulations \cite{balas1998disjunctive,ben2001lectures,ceria1999convex,helton2009sufficient,jeroslow1984modelling,stubbs1999branch} provide a \textit{sharp formulation} for a single disjunction, \ie the continuous relaxation provides the best possible lower bound. The convex hull is often represented by so-called extended (a.k.a.\ perspective/multiple-choice) formulations \cite{Balas2018,bonami2015mathematical,conforti2008compact,grossmann2003generalized,gunluk2010perspective,hijazi2012mixed,vielma2015mixed}, which introduce multiple copies of each variable in the disjunction(s). On the other hand, the big-M formulation only introduces one binary variable for each disjunct and results in a smaller problem in terms of both number of variables and constraints; however, in general it provides a weaker relaxation than the convex hull and may require a solver to explore significantly more nodes in a branch-and-bound tree \cite{conforti2014integer,vielma2015mixed}. Even though the big-M formulation is weaker, in some cases it can computationally outperform extended convex hull formulations, as the simpler subproblems can offset the larger number of explored nodes. 
Anderson et al.\ \cite{anderson2020strong} describe a folklore observation in mixed-integer programming (MIP) that extended convex hull formulations tend to perform worse than expected. 
The observation is supported by the numerical results in Anderson et al.\  \cite{anderson2020strong} and in this paper.

This paper presents a framework for generating formulations for disjunctions between the big-M and convex hull with the intention of combining the best of both worlds: a tight, yet computationally efficient, formulation.  
The main idea behind the novel formulations is partitioning the constraints of each disjunct and moving most of the variables out of the disjunction. Forming the convex hull of the resulting disjunctions results in a smaller problem, while retaining some features of the convex hull. We call the new formulation the $P$-split, as the constraints are split into $P$ parts. While many efforts have been devoted to computationally efficient convex hull formulations  \cite{balas1988convex,conforti2008compact,jeroslow1988simplification,sawaya2007computational,trespalacios2015algorithmic,vielma2019small,vielma2010mixed,vielma2011modeling} and techniques for deriving the convex hull of MIP problems  \cite{balas1985disjunctive,lasserre2001explicit,lovasz1991cones,ruiz2012hierarchy,sherali1990hierarchy}, our primary goal is not to generate the convex hull.
Rather, we provide a straightforward framework for generating a family of relaxations that approximate the convex hull for a general class of disjunctions using a smaller problem formulation. 
Our experiments show that the $P$-split formulations can give a significant computational advantage over both the big-M and convex hull formulations. 

This paper is organized as follows: 
the $P$-split formulation is presented in Section 2, together with properties of the $P$-split relaxations and how they compare to the big-M and convex hull relaxations. We also present a non-extended realization of the $P$-split formulation for the special case of a two-term disjunction. Finally, a numerical comparison of the formulations is presented in Section 3 using both instances with linear and nonlinear disjunctions.

\subsection{Background}
We consider optimization problems containing disjunctions of the form
\begin{equation}
\begin{aligned}
\label{eq:main_disjunction}
    &\underset{l \in \mathcal{D}}{\lor} \begin{bmatrix} g_{k}(\boldsymbol{x}) \leq b_k \quad \forall k \in \mathcal{C}_{l}
    \end{bmatrix}\\
    &\boldsymbol{x} \in \mathcal{X} \subset \mathbb{R}^n,
\end{aligned}
\end{equation}
where $\mathcal{D}$ contains the indices of the disjuncts, $\mathcal{C}_{l}$ the indices of the constraints in disjunct $l$, and $\mathcal{X}$ is a convex compact set. This paper assumes the following:
\begin{assumption}
The functions $g_{k}: \mathbb{R}^n \rightarrow \mathbb{R} $ are convex additively separable functions, \ie  $g_{k}(\boldsymbol{x}) = \sum_{i=1}^n h_{ik}(x_i)$ where $h_{ik}: \mathbb{R} \rightarrow \mathbb{R} $ are convex functions, and each disjunct is non-empty on $\mathcal{X}$.   
\end{assumption}
\begin{assumption}
All functions $g_k$ are bounded over $\mathcal{X}$.
\end{assumption}
\begin{assumption}
Each disjunct contains far fewer constraints than the number of variables in the disjunction, \ie $\left| \mathcal{C}_{l}\right| << n$.
\end{assumption}

The first two assumptions are needed for the  $P$-split formulation to be valid and result in a convex MIP.
While the first assumption simplifies our analysis of $P$-split formulations, it could easily be relaxed to partially additively separable functions. 
Furthermore, the computational experiments only consider problems with linear or quadratic constraints, which ensures that the convex hull of the disjunction is representable by a polyhedron or (rotated) second-order cone constraints \cite{ben2001lectures}. 
Assumption 3 characterizes problem structures favorable for the presented formulations. Problems with such a structure include, \eg  clustering \cite{papageorgiou2018pseudo,sauglam2006mixed}, mixed-integer classification \cite{liittschwager1978integer,rubin1997solving}, optimization over trained neural networks \cite{anderson2020strong,botoeva2020efficient,fischetti2018deep,grimstad2019relu,serra2020lossless}, and coverage optimization \cite{huang2005coverage}.

\section{Relaxations between convex hull and big-M}
The formulations in this section apply to disjunctions with multiple constraints per disjunct. However, to simplify the derivation, we only consider disjunctions with one constraint per disjunct, \ie  $|\mathcal{C}_{l}| = 1 \ \forall l \in  \mathcal{D}$. The extension to multiple constraints per disjunct simply applies the splitting procedure to each constraint. 

To derive the new formulations, we partition the variables into $P$ sets and form the corresponding index sets $\mathcal{I}_1, \dots, \mathcal{I}_P$. The constraint for each disjunct is then split into $P$ constraints, by introducing auxiliary variables $\alpha^j \in \mathbb{R}^{P}$
\begin{equation}
\label{eq:main_disjunction_lifted}
\begin{matrix}
\begin{aligned}
    &\underset{l \in \mathcal{D}}{\lor} \begin{bmatrix}  
    g_l(\boldsymbol{x}) \leq b_l\\
    \end{bmatrix}\\ 
    &\boldsymbol{x} \in \mathcal{X} 
    \end{aligned}
\end{matrix}
\quad \quad 
\longrightarrow  \quad \quad  
\begin{matrix}
\begin{aligned}
    &\underset{l \in \mathcal{D}}{\lor} 
    \begin{bmatrix}
    \begin{aligned}
    &\underset{i \in \mathcal{I}_1}{\sum}h_{i,l}(x_i) \leq \alpha^l_1\\[-0.1cm]
    & \quad \quad \quad \vdots \\
    \vspace{0.1cm}
    &\underset{i \in \mathcal{I}_P}{\sum}h_{i,l}(x_i) \leq \alpha^l_P\\
    \vspace{0.1cm}
    &\sum_{s=1}^P \alpha^l_s  \leq b_l\\
    & \ubar{\alpha}^l_s\leq \alpha^l_s \leq \bar{\alpha}^l_s  \quad \forall s \in \{1,\dots, P\}
    \end{aligned}
    \end{bmatrix} 
    \\
    &\boldsymbol{x} \in \mathcal{X}, \boldsymbol{\alpha}^l \in \mathbb{R}^{P}\ \forall\ l \in \mathcal{D}.  
    \end{aligned}
\end{matrix}
\end{equation}

By Assumption 2, function $h_{i,l}$ is bounded on $\mathcal{X}$, and bounds on the auxiliary variables are given by  
\begin{equation}
    \begin{aligned}
        & \ubar{\alpha}^l_s := \min_{\boldsymbol{x} \in \mathcal{X}} \underset{i \in \mathcal{I}_s}{\sum}h_{i,l}(x_i), \quad & \bar{\alpha}^l_s := \max_{\boldsymbol{x} \in \mathcal{X}} \underset{i \in \mathcal{I}_s}{\sum}h_{i,l}(x_i) .
    \end{aligned}
\end{equation}
The $P$-split formulation does not require tight bounds, but weak bounds result in an overall weaker relaxation. 

The splitting creates a lifted formulation by introducing $P \times |\mathcal{D}|$ auxiliary variables. Both formulations in~\eqref{eq:main_disjunction_lifted} have the same feasible set in the $\boldsymbol{x}$ variables. We relax the disjunction by treating the splitted constraints as global constraints 
\begin{equation}
\label{eq:main_disjunction_splitted}
\begin{aligned}
    &\underset{l \in \mathcal{D}}{\lor} \begin{bmatrix}
    \begin{aligned}
    &\sum_{s=1}^P \alpha^l_s  \leq b_l\\
    &\ubar{\alpha}^l_s\leq \alpha^l_s \leq \bar{\alpha}^l_s  \quad \forall s \in \{1,\dots, P\}
    \end{aligned}
    \end{bmatrix} 
    \\
    &\underset{i \in \mathcal{I}_s}{\sum}h_{i,l}(x_i) \leq \alpha^l_s  &&\forall s \in \{1,\dots, P\}, \ \forall \ l \in \mathcal{D} \\
    &\boldsymbol{x} \in \mathcal{X}, \boldsymbol{\alpha}^l \in \mathbb{R}^{P}\ &&\forall \ l \in \mathcal{D}.  
    \end{aligned}
\end{equation}
\begin{definition}
Formulation \eqref{eq:main_disjunction_splitted} is a $P$-split representation of the original disjunction in \eqref{eq:main_disjunction_lifted}.
\end{definition}
Lemma \ref{lemma_feasset} relates the $P$-split representation to the original disjunction. The property is rather simple, but for completeness we have stated it as a lemma.
\begin{lemma}\label{lemma_feasset}
The feasible set of $P$-split representation projected onto the $\boldsymbol{x}$-space is equal to the feasible set of the original disjunctions in \eqref{eq:main_disjunction_lifted}. 
\end{lemma}
\begin{proof}
An $\bar{\boldsymbol{x}}$ that is feasible for \eqref{eq:main_disjunction_splitted} and violates \eqref{eq:main_disjunction_lifted} gives a contradiction. Similarly, an $\bar{\boldsymbol{x}}$ that is feasible for \eqref{eq:main_disjunction_lifted} is also clearly feasible for \eqref{eq:main_disjunction_splitted}. \qed
\end{proof}

Using the extended formulation \cite{balas1998disjunctive} to represent the convex hull of the disjunction in \eqref{eq:main_disjunction_splitted} results in the \textit{$P$-split formulation}
\begin{equation}
\label{eq:p-split}
\tag{$P$-split}
\begin{aligned}
& \alpha^l_s = \underset{d \in \mathcal{D}}{\sum} \nu^{\alpha^l_s}_d && \forall \ s \in \{1, \dots, P\}, \ \forall\ l \in \mathcal{D} \\
& \sum_{s=1}^P \nu^{\alpha^l_s}_l  \leq b_l\lambda_l &&\forall\ l \in \mathcal{D}\\
& \ubar{\alpha}^l_s\lambda_d \leq \nu^{\alpha^l_s}_d \leq \bar{\alpha}^l_s\lambda_d  &&\forall \ s \in \{1, \dots, P\},\forall\ l,d \in \mathcal{D}\\
 &\underset{i \in \mathcal{I}_s}{\sum}h_{i,l}(x_i) \leq \alpha^l_s  &&\forall\ s \in \{1,\dots, P\}, \ \forall \ l \in \mathcal{D} \\
 & \underset{l \in \mathcal{D}}{\sum}\lambda_l = 1, \quad   \boldsymbol{\lambda} \in \{0, 1\}^{|\mathcal{D}|}\\
&\boldsymbol{x} \in \mathcal{X}, \boldsymbol{\alpha}^l \in \mathbb{R}^{P},  \ \boldsymbol{\nu}^{\alpha^l_s} \in \mathbb{R}^P &&\forall\ s \in \{1,\dots, P\}, \ \forall \ l \in \mathcal{D}\ ,
\end{aligned}
\end{equation}
which forms a convex MIP problem. To clarify our terminology: a 2-split formulation is a formulation \eqref{eq:p-split} where the constraints of the original disjunction are split up into two parts, \ie $P = 2$. We assume that the disjunction is part of a larger optimization problem that may contain multiple disjunctions. Therefore, we need to enforce integrality on the $\lambda$ variables even if we recover the convex hull of the disjunction. Proposition \ref{prop_feasset} shows the correctness of the the \eqref{eq:p-split} formulation of the original disjunction.
\begin{proposition}\label{prop_feasset}
The set of feasible $\boldsymbol{x}$ variables in formulation ~\eqref{eq:p-split} is equal to the feasible set of $\boldsymbol{x}$ variables in disjunction~\eqref{eq:main_disjunction_lifted}.
\end{proposition}
\begin{proof}
By Lemma 1, \eqref{eq:main_disjunction_lifted} and \eqref{eq:main_disjunction_splitted} have equivalent $\boldsymbol{x}$ feasible sets. For $\lambda \in \{0, 1\}^{|\mathcal{D}|}$, the extended formulation \eqref{eq:p-split} exactly represents the disjunction \eqref{eq:main_disjunction_splitted}. \qed
\end{proof}
Proposition \ref{prop_feasset} states that the $P$-split formulation is correct for integer feasible solutions, but it does not give any insight on the quality of the continuous relaxation. The following subsections further analyze the properties of the \eqref{eq:p-split} formulation and its relation to the big-M and convex hull formulations. 

\begin{remark}
A \eqref{eq:p-split} formulation introduces $P\cdot \left(|\mathcal{D}|^2 +1 \right)$ continuous  and $|\mathcal{D}|$ binary variables. Unlike the extended convex hull formulation (which introduces $|\mathcal{D}|\cdot n$ continuous and $|\mathcal{D}|$ binary variables), the number of \say{extra} variables is independent of $n$, \ie the number of variables in the original disjunction. 
As we later show, there are applications where $|\mathcal{D}| << n$ for which \eqref{eq:p-split} formulations can be smaller and computationally more tractable than the extended convex hull formulation.
\end{remark}

\subsection{Properties of the $P$-Split formulation}
This section focuses on the strength of the continuous relaxation of the $P$-split formulation, and how it compares to convex hull and big-M formulations. To simplify the analyses, we only consider disjunctions with a single constraint per disjunct. However, the results directly extend to the case of multiple constraints per disjunct by applying the same procedure to each individual constraint.  

We first analyze the 1-split, as summarized in the following theorem.

\begin{theorem}
The 1-split formulation is equivalent to the big-M formulation.
\end{theorem}
\begin{proof} 
We eliminate the disaggregated variables $\nu^{\alpha^l}_d$ from the 1-split formulation using Fourier-Motzkin elimination. 
Furthermore, we eliminate trivially redundant constraints, \eg  $\ubar{\alpha}^l\lambda_d \leq \bar{\alpha}^l\lambda_d$, resulting in
\begin{equation}
\label{eq:1-split_bigM}
\begin{aligned}
&  \alpha^l  \leq b_l\lambda_l +  \underset{d \in \mathcal{D} \setminus l}{\sum}\bar{\alpha}^l\lambda_d  && \forall l \in \mathcal{D} \\
  &\sum_{i=1}^nh_{i,l}(x_i) \leq \alpha^l  &&\forall \ l \in \mathcal{D} \\
 & \underset{l \in \mathcal{D}}{\sum}\lambda_l = 1, \quad  \boldsymbol{\lambda} \in \{0, 1\}^{|\mathcal{D}|},\boldsymbol{x} \in \mathcal{X}, \boldsymbol{\alpha}^l \in \mathbb{R} \ &&\forall \ l \in \mathcal{D}.
\end{aligned}
\end{equation}
The auxiliary variables $ \alpha^l$ are removed by combining the first and second constraints in \eqref{eq:1-split_bigM}. The smallest valid big-M coefficients are $M^l = \bar{\alpha}^l - b_l $, which enables us to write \eqref{eq:1-split_bigM} as
\begin{equation}
\begin{aligned}
&  \sum_{i=1}^nh_{i,l}(x_i)    \leq b_l  + M^l(1- \lambda_l)  && \forall l \in \mathcal{D}_k \\
 & \underset{l \in \mathcal{D}}{\sum}\lambda_l = 1, \quad  \boldsymbol{\lambda} \in \{0, 1\}^{|\mathcal{D}|},\ \boldsymbol{x} \in \mathcal{X}.
\end{aligned}
\end{equation}
\qed
\end{proof}
Since  the 1-split formulation introduces $|\mathcal{D}|^2 + 1$ auxiliary variables, but has the same continuous relaxation as the big-M formulation, there are no clear advantages of the 1-split formulation vs the big-M formulation. 

We now examine the other extreme, where constraints are fully disaggregated, \ie the $n$-split. Its relation to the convex hull is given in the following theorem.

\begin{theorem}
If all $h_{i,l}$ are affine functions, then the $n$-split formulation (where constraints are split for each variable) provides the convex hull of the disjunction.
\end{theorem}
\begin{proof}
In the linear case, the original disjunction is given by
\begin{equation}
\label{eq:lin_disjunct}
    \begin{matrix}
\begin{aligned}
    &\underset{l \in \mathcal{D}}{\lor} \begin{bmatrix}  
    (\mathbf{a}^l)^T \boldsymbol{x} \leq {b_l}\\
    \end{bmatrix}\\ 
    &\boldsymbol{x} \in \mathcal{X},
    \end{aligned}
\end{matrix}
\end{equation}
and the $n$-split formulation can be written compactly as 
\begin{equation}
\label{eq:p-split-rep-lin}
    \begin{matrix}
\begin{aligned}
    &\underset{l \in \mathcal{D}}{\lor} \begin{bmatrix}  
    \mathbf{B}^l\tilde{\boldsymbol{\alpha}} \leq \tilde{\boldsymbol{b}_l}\\
    \end{bmatrix}\\ 
    &\tilde{\boldsymbol{\alpha}} = \mathbf{\Gamma} \boldsymbol{x}, \quad \boldsymbol{x} \in \mathcal{X}, \tilde{\boldsymbol{\alpha}} \in \mathbb{R}^{n\times|\mathcal{D}|}.
    \end{aligned}
\end{matrix}
\end{equation}
The $n$-split formulation is given by the convex hull of \eqref{eq:p-split-rep-lin} through the extended formulation. Here, $\mathbf{\Gamma}$ defines a bijective mapping between the   $\boldsymbol{x}$ and $\tilde{\boldsymbol{\alpha}}$ variable spaces (only true for an $n$-split). A reverse mapping is given by $\boldsymbol{x}= \mathbf{\Psi}\tilde{\boldsymbol{\alpha}}$. The linear transformations preserve an exact representation of the feasible sets, \ie
\begin{equation}
\begin{aligned}
  &  \mathbf{B}^l\tilde{\boldsymbol{\alpha}} \leq \tilde{\boldsymbol{b}_l} \iff (\mathbf{a}^l)^T \mathbf{\Psi}\tilde{\boldsymbol{\alpha}} \leq b_,  \quad \quad (\mathbf{a}^l)^T \boldsymbol{x} \leq b_l  \iff \mathbf{B}^l\mathbf{\Gamma} \boldsymbol{x} \leq \tilde{\boldsymbol{b}_l}.    
\end{aligned}
\end{equation}
For any point $\boldsymbol{z}$ in the the convex hull of  \eqref{eq:p-split-rep-lin} $\exists\ \tilde{\boldsymbol{\alpha}}^1, \tilde{\boldsymbol{\alpha}}^2, \dots \tilde{\boldsymbol{\alpha}}^{\mathcal{|D}|}$ and $\boldsymbol{\lambda} \in \mathbb{R}_+^{|\mathcal{D}|}$
\begin{align}
\label{eq:convex-comb}
    &\boldsymbol{z} = \sum_{l=1}^{|\mathcal{D}|}\lambda_l\tilde{\boldsymbol{\alpha}}^l\\
    &  \sum_{l=1}^{|\mathcal{D}|}\lambda_l = 1,\ \  \mathbf{B}^l\tilde{\boldsymbol{\alpha}}^l \leq \tilde{\boldsymbol{b}_l} \quad  \forall\ l \in \mathcal{D} \nonumber.
\end{align}
Applying the reverse mapping to \eqref{eq:convex-comb} gives
\begin{equation}
     \mathbf{\Psi}\boldsymbol{z} = \sum_{l=1}^{|\mathcal{D}|}\lambda_l\mathbf{\Psi}\tilde{\boldsymbol{\alpha}}^l.
\end{equation}
By construction, $ \left(\mathbf{a}^l\right)^T\mathbf{\Psi}\tilde{\alpha}^l \leq b_l \quad  \forall l \in \mathcal{D}$.  The point $\mathbf{\Psi}\boldsymbol{z}$ is given by a convex combination of points that all satisfy the constraints of one of the disjuncts in \eqref{eq:lin_disjunct} and, therefore, belongs to the convex hull of \eqref{eq:lin_disjunct}. The same technique easily shows that any point in the convex hull of disjunction~\eqref{eq:lin_disjunct} also belongs to the convex hull of disjunction~\eqref{eq:p-split-rep-lin}.\qed
\end{proof}
Theorem 2 does not hold with nonlinear functions, since the mapping may not be bijective or a homomorphism. In general, the $n$-split formulation will not obtain the convex hull of nonlinear disjunctions, as Section 2.2 shows by example, but it can provide a strong outer approximation. 

\subsubsection{Two-term disjunctions}
We further analyze the special case of a two-term disjunction for which we also present a non-lifted $P$-split formulation in the following theorem. 
\begin{theorem}
For a  two-term disjunction, the $P$-split formulation has the following non-extended realization 
\begin{equation}
\begin{aligned}
\label{eq:P-2-split-nonlifted}
& \underset{j \in \mathcal{S}_p}{\sum}\left( \underset{i \in \mathcal{I}_j}{\sum}h_{i,1}(x_i) \right)\leq \left(b_1 - \underset{s \in \mathcal{S}\setminus \mathcal{S}_p}{\sum} \ubar{\alpha}^1_s\right)\lambda_1 + \underset{s \in \mathcal{S}_p}{\sum}\bar{\alpha}^1_s\lambda_2 && \forall \mathcal{S}_p \subset \mathcal{S}\\
& \underset{j \in \mathcal{S}_p}{\sum}\left( \underset{i \in \mathcal{I}_j}{\sum}h_{i,2}(x_i) \right)\leq \left(b_2 - \underset{s \in \mathcal{S}\setminus \mathcal{S}_p}{\sum} \ubar{\alpha}^2_s\right)\lambda_2 + \underset{s \in \mathcal{S}_p}{\sum}\bar{\alpha}^2_s\lambda_1 && \forall \mathcal{S}_p \subset \mathcal{S}\\
 & \lambda_1 + \lambda_2  = 1, \ \ \boldsymbol{\lambda} \in \{0, 1\}^2, \ \ \boldsymbol{x} \in \mathcal{X},
\end{aligned}
\end{equation}
where $\mathcal{S} = \{1, 2, \dots P\}$.
\end{theorem}
\begin{proof}
The equality constraints for the disaggregated variables ($\alpha_s^l = \nu_1^{\alpha_s^l} + \nu_2^{\alpha_s^l}$) enable us to easily eliminate the variables $\nu^{\alpha^l_s}_1$ from \eqref{eq:p-split}, resulting in
\begin{align}
& \sum_{s=1}^P\left( \alpha^1_s - \nu^{\alpha^1_s}_2\right)  \leq b_1\lambda_1 \label{eq:p-2-s-eq1}\\
& \sum_{s=1}^P \nu^{\alpha^2_s}_2  \leq b_2\lambda_2 \label{eq:p-2-s-eq2}\\
& \ubar{\alpha}^l_s\lambda_1 \leq \alpha^l_s - \nu^{\alpha^l_s}_2 \leq \bar{\alpha}^l_s\lambda_1  &&\forall s  \in \{1, 2 , \dots, P\},\forall\ l \in  \{1, 2\}\label{eq:p-2-s-eq3}\\
& \ubar{\alpha}^l_s\lambda_2 \leq \nu^{\alpha^l_s}_2 \leq \bar{\alpha}^l_s\lambda_2  &&\forall s  \in \{1, 2 , \dots, P\},\forall\ l \in  \{1, 2\}\label{eq:p-2-s-eq4}\\
 &\underset{i \in \mathcal{I}_s}{\sum}h_{i,l}(x_i) \leq \alpha^l_s  &&\forall\ s \in \{1, 2 , \dots, P\}, \ \forall \ l \in  \{1, 2\} \label{eq:p-2-s-eq5}\\
 & \lambda_1 + \lambda_2  = 1, \quad  \boldsymbol{\lambda} \in \{0, 1\}^2 \label{eq:p-2-s-eq6}\\
&\boldsymbol{x} \in \mathcal{X}, \boldsymbol{\alpha}^l \in \mathbb{R}^{P}, \boldsymbol{\nu}^{\alpha^l_s} \in \mathbb{R}^P\ &&\forall \ l \in \{1, 2\}, \forall\ s \in \{1, 2 , \dots, P\}.    
\end{align}
Next, we use Fourier-Motzkin elimination to project out the $\nu^{\alpha^1_s}_2$ variables. Combining the constraints in \eqref{eq:p-2-s-eq3} and \eqref{eq:p-2-s-eq4} only results in trivially redundant constraints, \eg  $\alpha^l_s \leq \bar{\alpha}^l_s(\lambda_1 + \lambda_2)$. Eliminating the first variable $\nu^{\alpha^1_1}_2$ creates two new constraints by combining \eqref{eq:p-2-s-eq1} with \eqref{eq:p-2-s-eq3}--\eqref{eq:p-2-s-eq4}. The first constraint is obtained by removing $\nu^{\alpha^1_1}_2$ and $\alpha^1_1$ from \eqref{eq:p-2-s-eq1} and adding $\ubar{\alpha}^1_1\lambda_2$ to the left-hand side. The second constraint is obtained by removing $\nu^{\alpha^1_1}_2$ from \eqref{eq:p-2-s-eq1} and subtracting $\bar{\alpha}^1_1\lambda_2$ from the left-hand side. Eliminating the next variable is done by repeating the procedure of combining the two new constraints with the corresponding inequalities in \eqref{eq:p-2-s-eq3}--\eqref{eq:p-2-s-eq4}. Each elimination step doubles the number of constraints originating from inequality \eqref{eq:p-2-s-eq1}. Eliminating all the variables $\nu^{\alpha^1_s}_2$ and $\alpha^1_s$ results in the first set of constraints 
\begin{equation}
   \underset{s \in \mathcal{S}_p}\sum \alpha^1_s \leq \left(b_1 - \underset{s \in \mathcal{S}\setminus \mathcal{S}_p}{\sum} \ubar{\alpha}^1_s\right)\lambda_1 + \underset{s \in \mathcal{S}_p}{\sum}\bar{\alpha}^1_s\lambda_2 \quad \forall \mathcal{S}_p \subset \mathcal{S}.
\end{equation}
The variables $\nu^{\alpha^2_s}_2$ and $\alpha^2_s$ are eliminated by same steps, resulting in the second set of constraints in \eqref{eq:P-2-split-nonlifted}. \qed
\end{proof}

To further analyze the tightness of different $P$-split relaxations we require that the bounds on the auxiliary variables be \textit{independent}, as defined below:
\begin{definition}
We say that the upper and lower bounds for the constraint\\ $\sum_{i=1}^n h_{i}(x_i) \leq 0$ are independent on $\mathcal{X}$ if 
\begin{equation}
\begin{aligned}
  &\min_{\boldsymbol{x} \in \mathcal{X}} \left( h_{i}(x_i) + h_{j}(x_j) \right) =  \min_{\boldsymbol{x} \in \mathcal{X}} \ h_{i}(x_i) +\min_{\boldsymbol{x} \in \mathcal{X}} \ h_{j}(x_j) \\
  &\max_{\boldsymbol{x} \in \mathcal{X}} \left( h_{i}(x_i) + h_{j}(x_j) \right) =  \max_{\boldsymbol{x} \in \mathcal{X}} \ h_{i}(x_i) +\max_{\boldsymbol{x} \in \mathcal{X}} \ h_{j}(x_j),
\end{aligned}
\end{equation}
hold for all $i,j \in \{1, 2, \dots n\}$.

\end{definition}
Independent bounds are not restricted to linear constraints, but the most general case of independent bounds are linear disjunctions with $\mathcal{X}$ defined as a box. Independent bounds enable us to establish a strict relation on the tightness of different $P$-split formulations, which is presented in the next corollary.   
\begin{corollary}
For a two-term disjunction with independent bounds, a $(P+1)$-split formulation, obtained by splitting one variable group in the $P$-split, is always as tight or tighter than the corresponding P-split formulation.  
\end{corollary}
\begin{proof}
The non-extended formulation \eqref{eq:P-2-split-nonlifted} for the $(P+1)$-split comprises the constraints in the $P$-split formulation and some additional constraints. \qed
\end{proof}
From Corollary 1 it follows that the $P$-split formulations represent a hierarchy of relaxations, and we formally state this property in the following corollary.
\begin{corollary}
For a linear two-term disjunction the P-split formulations form a hierarchy of relaxations, starting from the big-M relaxation ($P=1$) and converging to the convex hull relaxation ($P=n$). 
\end{corollary}
\begin{proof}
Theorems 1 and 2 give equivalence to big-M and convex hull. By Corollary 1, the $(P+1)$-split is as tight or tighter than the $P$-split relaxation. \qed
\end{proof}

\subsection{Illustrative example}
To see the differences between $P$-split formulations, consider the disjunction
\begin{equation}
\label{eq:example1}
\tag{ex-1}
\begin{aligned}
    &\begin{bmatrix}
    \sum_{i=1}^4 x_i ^2  \leq 1 
    \end{bmatrix}
    \
    \lor \
    \begin{bmatrix}
    \sum_{i=1}^4 (3-x_i)^2 \leq 1
    \end{bmatrix}\\
    & \boldsymbol{x} \in \mathbb{R}^4.    
\end{aligned}
\end{equation}
The tightest valid bounds on all the auxiliary variables are given by
\begin{equation}
    \ubar{\alpha}^l_s = 0,\quad  \bar{\alpha}^l_s := \left(\sqrt{|\mathcal{I}_s|\cdot3^2} + 1 \right)^2 \quad \forall s \in \{1,2,3,4\},\ \forall l \in \{1,2\}.
\end{equation}
These bounds are derived from the fact that one of the two constraints in the disjunction must hold, and are symmetric for the two set of $\alpha$-variables. The continuously relaxed feasible sets of the $P$-split formulations of disjunction \eqref{eq:example1} are shown in Fig.~\ref{fig:relaxations}, which shows that the relaxations overall tighten with increasing number of splits $P$. The 4-split formulation does not give the convex hull, but provides a good approximation. For this example, the independent bound property does not hold and the relaxations do not form a proper hierarchy. To show why the independent bound property is needed, we compare the non-extended representations of the 1-split and 2-split formulations:
\begin{align}
    &\sum_{i=1}^4 x_i^2 \leq \lambda_1 + \left(\sqrt{36} +1\right)^2\lambda_2, &&\sum_{i=1}^4 (3-x_i)^2  \leq \lambda_2 + \left(\sqrt{36} +1\right)^2\lambda_1 \tag{1-s}\\
    &  \sum_{i=1}^2 x_i ^2 \leq \lambda_1 +\left(\sqrt{18} +1\right)^2\lambda_2, &&\sum_{i=3}^4 x_i ^2 \leq \lambda_1 +\left(\sqrt{18} +1\right)^2\lambda_2 \tag{2-s1}\\
    & \sum_{i=1}^2 (3-x_i)^2  \leq \lambda_2 +\left(\sqrt{18} +1\right)^2\lambda_1,   \hspace{-0.4cm}&&\sum_{i=3}^4 (3-x_i)^2  \leq \lambda_2 + \left(\sqrt{18} +1\right)^2\lambda_1 \tag{2-s2}\\
    &  \sum_{i=1}^4 x_i ^2  \leq \lambda_1 + 2\left(\sqrt{18} +1\right)^2\lambda_2,  &&\hspace{-0.4cm}\sum_{i=1}^4 (3-x_i)^2 \leq \lambda_2 +2\left(\sqrt{18} +1\right)^2\lambda_1. \tag{2-s3}
\end{align}
The 1-split formulation is given by (1-s), and the 2-split by (2-s1)--(2-s3). 
The 2-split contains additional constraints ~(2-s1) and (2-s2), but ~(2-s3) is a weaker version of (1-s). If the independent bound property were true, then (2-s3) and (1-s) would be identical and the relaxations would form a proper hierarchy. 
\vspace{-10pt}
\begin{figure}[!h]
    \centering
    \begin{subfigure}[t]{.32\linewidth}
    \includegraphics[trim={1.5cm 0cm 2cm 0cm},clip,width=.99\linewidth]{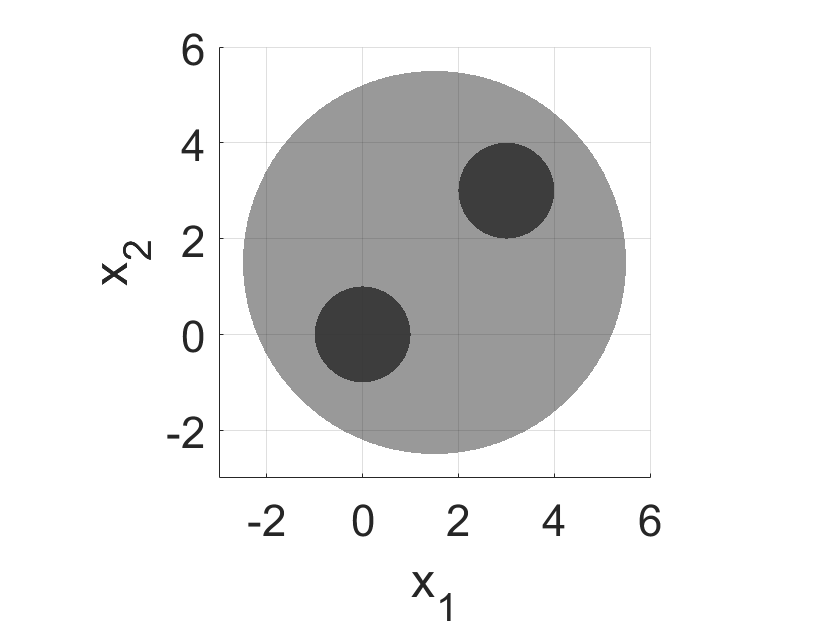}
    \caption{\centering 1-split/big-M $\left(\{x_1, x_2, x_3, x_4\}\right)$}
    \end{subfigure}
    \begin{subfigure}[t]{.32\linewidth}
    \includegraphics[trim={1.5cm 0cm 2cm 0cm},clip,width=.99\linewidth]{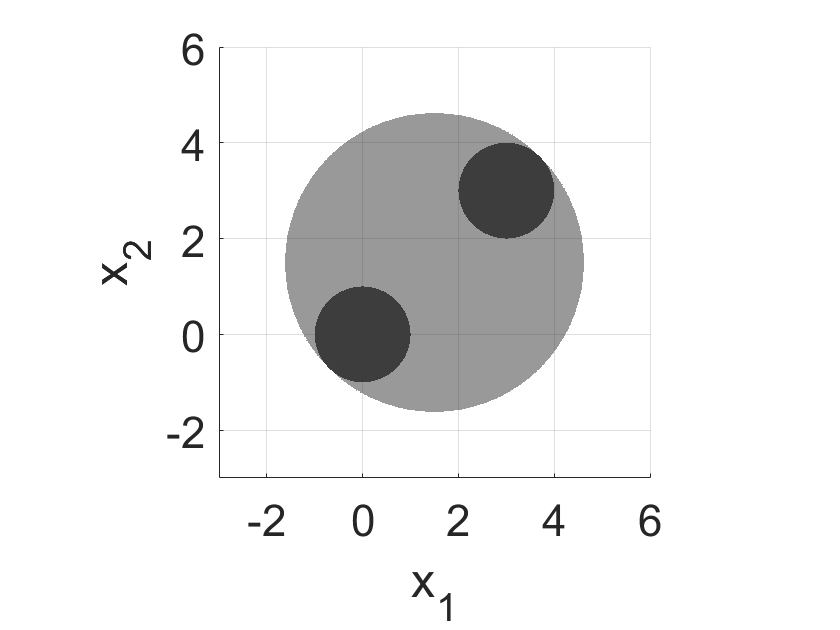}
    \caption{\centering 2-split \break $\left(\{x_1, x_2\}, \{x_3, x_4\}\right)$}
    \end{subfigure}
    \begin{subfigure}[t]{.32\linewidth}
    \includegraphics[trim={1.5cm 0cm 2cm 0cm},clip,width=.99\linewidth]{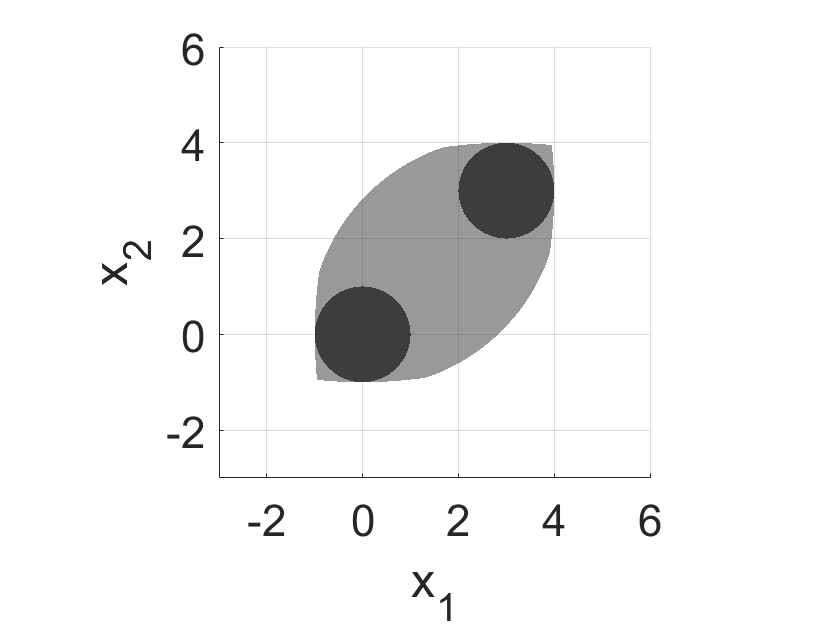}
    \caption{\centering 4-split $\left(\{x_1\},\{x_2\},\{x_3\},\{x_4\}\right)$}
    \end{subfigure}
      \caption{The dark circles show the feasible set of \eqref{eq:example1} in the $x_1,x_2$ space. The light grey areas show the continuously relaxed feasible set of the P-split formulations. The sets in the parentheses show the partitioning of variables.}
    \label{fig:relaxations}
\end{figure}

\section{Numerical comparison}
To compare how the formulations perform computationally, we apply the $P$-split, big-M, and convex hull formulations to several test problems. We consider three types of optimization problems that have a suitable structure for the $P$-split formulation (assumptions 1--3) and that are known to be challenging. 
\subsubsection{K-means clustering} 
Using the formulation by Papageorgiou and Trespalacios \cite{papageorgiou2018pseudo}, the K-means clustering problem \cite{macqueen1967some} is given by
\begin{equation}
\label{eq:clustering}
    \begin{aligned}
        & \min_{\mathbf{r} \in \mathbb{R}^L, \boldsymbol{x}^j \in \mathbb{R}^n, \forall j \in \mathcal{K}} &&\sum_{i=1}^L r_i\\
        & \text{s.t.} &&\underset{j \in \mathcal{K}}{\lor} \begin{bmatrix}
        \norm{\boldsymbol{x}^j - \mathbf{d}^i}_2^2 \leq r_i  
        \end{bmatrix}\quad  \forall i \in \{1, 2, \dots, L\},\\
    \end{aligned}
\end{equation}
where $\boldsymbol{x}^j$ are the cluster centers, $\{\mathbf{d}^i\}_{i=1}^L$ are $n$-dimensional data points, and $\mathcal{K} =\{1,2,\dots k\}$. The tightest upper bound for the auxiliary variables in the P-split formulations are given by the largest squared Euclidean distance between any two data points in the subspace corresponding to the auxiliary variable. By introducing auxillary variables for the differences $(\boldsymbol{x}-\mathbf{d})$, we can express the convex hull of the disjunctions by rotated second order cone constraints \cite{ben2001lectures} in a form suitable for Gurobi. We use the G2 data set \cite{G2sets} to generate low-dimensional test instances, and the MNIST data set \cite{lecun2010mnist} to generate high-dimensional test instances. For the MNIST-based problems, we select the first images of each class ranging from 0 to the number of clusters. Details about the problems are presented in Table~\ref{tab:test_probs}. 

\subsubsection{P\_ball problems} The task is to assign $p$-points to $n$-dimensional unit balls such that the total $\ell_1$ distance between all points is minimized and only one point is assigned to each unit ball \cite{kronqvist2020disjunctive}. Upper bounds on the auxiliary variables in the P-split formulation are given by the same technique as for the $M$-coefficients in \cite{kronqvist2020disjunctive}, but in the subspace corresponding to the auxiliary variable. By introducing auxiliary variables for the differences between the points and the centers, we are able to express the convex hull by second order cone constraints \cite{ben2001lectures} in a form suitable for Gurobi. We have generated a few larger instances to obtain more challenging problems and details of the problems are given in Table~\ref{tab:test_probs}. 

\subsubsection{ReLU neural networks}
Optimization over a ReLU neural network (NN) is used to quantify extreme outputs \cite{anderson2020strong,botoeva2020efficient}. Each ReLU activation function ($y = \textrm{max} \{ 0, \boldsymbol{w}^T \boldsymbol{x} + b\}$) can be expressed as a two-part disjunction using the $P$-split formulation, by separating $\boldsymbol{w}^T \boldsymbol{x} = \sum_{i \in \mathcal{S}_1 \cup ... \cup \mathcal{S}_P} w_i x_i$. 
We sort the variables $x_i$ by index and assign them to splits of even size. 
Upper bounds on node outputs and auxiliary variables can be computed using simple interval arithmetic.
We created several instances (Table~\ref{tab:test_probs}) that minimize the prediction of single-output NNs trained on the $d$-dimensional Ackley/Rastrigin functions. 
All NNs were implemented in PyTorch \cite{pytorch} and trained for 1000 epochs, using a Latin hypercube of 10$^6$ samples. 
Note that more samples may be required to accurately represent the target functions, but here we are solely concerned with the performance of various optimization formulations.

\begin{table}[]
\caption{Details of the clustering, P$\_$ball and neural network problems.}
\centering
\begin{tabular}{ c| c | c |c } \hline
 name &\ data points \ & \ data dimension \ & \ number of clusters\\
 \hline
 Cluster\_g1 & 20 & 32 & 2\\ 
 Cluster\_g2 & 25 & 32 & 2\\ 
 Cluster\_g3 & 20 & 16 & 3\\ \hline
 Cluster\_m1 & 5 & 784 & 3\\ 
 Cluster\_m2 & 8 & 784 & 2\\ 
 Cluster\_m3 & 10 & 784 & 2\\
 \hline
  &\ number of balls \ & \ number of points \ & \ ball dimension\\
 \hline
 P\_ball\_1\ \ & 10 & 5 & 8\\ 
 P\_ball\_2\ \ & 10 & 5 & 16\\ 
 P\_ball\_3\ \ & 8 & 5 & 32\\
\hline
  &\ input dimension ($d$) \ & \ hidden layers \ & \ function \\
  \hline
 NN\_1 & 2 & [50, 50, 50] & Ackley \\
 NN\_2 & 10 & [50, 50, 50] & Ackley \\
 NN\_3 & 3 & [100, 100] & Rastrigin \\ \hline
\end{tabular}
\label{tab:test_probs}
\vspace{-8pt}
\end{table}

\subsubsection{Computational setup}
Optimization performance is dependent on both the tightness and the computational complexity of the continuous relaxation. 
The default (automatic) parameter selection in Gurobi caused large variations in the results that were due to different solution strategies rather than differences between formulations. 
Therefore, we used the parameter settings \texttt{MIPFocus} = 3, \texttt{Cuts} = 1, and \texttt{MIQCPMethod} = 1 for all problems. 
We found that using \texttt{PreMIQCPForm} = 2 drastically improves the performance of the extended convex hull formulations for the clustering and P\_ball problems. However, it resulted in worse performance for the other formulations and, therefore, we only used it with the convex hull. Since the NN problems only contain linear constraints, only the \texttt{MIPFocus} and \texttt{Cuts} parameters apply to these problems 
The default values were used for all other settings. 
All problems were solved using Gurobi 9.0.3 on a desktop computer with an i7 8700k processor and 16GB RAM. 

Different variable partitionings can lead to differences in the $P$-split formulations. For all the problems, the variables are simply partitioned based on their ordered indices. For the K-means clustering and P\_ball problems, we have used the smallest valid M-coefficients and thight bounds for the $\alpha$-variables. The K-means clustering and P\_ball problems both have analytical expressions for all the bounds. For the NN problems tight bounds are not easily obtained, and the bounds are obtained using interval arithmetic.
\subsection{Numerical results}
Table~\ref{tab:res} shows the elapsed CPU time and number of nodes explored to solve each problem. The results show that $P$-split formulations can drastically reduce the number of explored nodes compared to the big-M formulation, even with only a few splits. The differences are clearest for the nonlinear problems, where both the CPU times and numbers of nodes are reduced by several orders of magnitude. As expected, the convex hull formulation results in the fewest explored nodes. 
However, the $P$-split formulations have a simpler\footnotemark\  problem formulation, reducing the CPU times for all but one instance compared to the convex hull. The results clearly show the advantage of the intermediate $P$-split formulations, resulting in a tighter formulation than big-M and a computationally cheaper formulation than the extended convex hull. 
\begin{table}[!h]
\caption{CPU times [s] and numbers of nodes explored for test problems. In bold is the \emph{winner} for each test instance with respect to both time and number of nodes. The grey shading shows the $P$-split times that strictly outperform both the big-M and convex hull formulations. The time limit was 1800 CPU seconds. 
Cells marked NA correspond to instances with fewer than $P$ terms per disjunction.}
\vspace{2pt}
\centering
\begin{tabular}{ c c | c | c |  c | c | c | c | c } 
\hline
 instance & &\ big-M \ & \ 2-split \ & \ 4-split & \ 8-split & \ 16-split & \ 32-split & \ convex hull\\
 \hline
 Cluster\_g1 \ &time & $>$1800 & 81.0  & \cellcolor[gray]{0.8}13.9 & \cellcolor[gray]{0.8}2.9 & \cellcolor[gray]{0.8}\textbf{1.7} & \cellcolor[gray]{0.8}3.5 & 42.0\\
   & nodes & $>$8998 & 2946 & 1096 & 256 & 98 & 91 & \textbf{73}\\
    \hline
 Cluster\_g2 \ &time & $>$1800  & 106.3  & \cellcolor[gray]{0.8}7.7  & \cellcolor[gray]{0.8}4.3  & \cellcolor[gray]{0.8}\textbf{2.1 } & \cellcolor[gray]{0.8}4.5  & 40.6 \\
   & nodes & $>$10431 & 1736 & 481 & 217 & 104 & 86 & \textbf{77}\\
    \hline
 Cluster\_g3 \ &time & $>$1800 & $>$1800 & \cellcolor[gray]{0.8}870.6  &\  \cellcolor[gray]{0.8}\textbf{407.2}  & \cellcolor[gray]{0.8}597.5  & NA  & $>$1800\\
   & nodes & $>$28906 & $>$40820 & 19307 & \textbf{14923} & 16806 &   & $>$7797\\
    \hline
 P\_ball\_1 \ &time & 403.0 & 235.4 & 285.1 & \cellcolor[gray]{0.8}\textbf{18.5} &NA  &NA  & 42.2 \\
   & nodes & 29493 & 7919 & 5518 & 2202 &  &  & \textbf{1437}\\
    \hline
 P\_ball\_2 \ &time & $>$1800 & 483.6  & 326.6  & 41.6  & 30.6  &NA  & \textbf{28.2} \\
   & nodes & $>$19622 & 13602 & 5871 & 3921 & 1261 &  & \textbf{531}\\
    \hline
 P\_ball\_3 &time & $>$1800 & $>$1800& $>$1800 & 149.3 & \cellcolor[gray]{0.8}91.1 & \cellcolor[gray]{0.8}\textbf{78.7} & 114.0 \\
   & nodes & $>$7537 & $>$6035 & $>$6708 & 7042 & 3572 & 631 & \textbf{554} \\
    \hline
 & &\ big-M \ & \ 14-split \ & \ 28-split & \ 56-split & \ 196-split & \ 392-split & \ convex hull\\
 \hline
 Cluster\_m1 \ &time & $>$1800 & $>$1800  & \cellcolor[gray]{0.8}129.5 & \cellcolor[gray]{0.8}76.8 & \cellcolor[gray]{0.8}\textbf{32.0} & \cellcolor[gray]{0.8}33.2 & 313.3\\
   & nodes & $>$10680 & $>$9651 & 2926 & 1462 & 524 & \textbf{195} & 228\\
    \hline
 Cluster\_m2 \ &time & $>$1800  &\cellcolor[gray]{0.8} 1116.5  & \cellcolor[gray]{0.8} 156.1  & \cellcolor[gray]{0.8}\textbf{27.1}  & \cellcolor[gray]{0.8} 97.0 & \cellcolor[gray]{0.8}54.2  & 1260.1 \\
   & nodes & $>$4867 & 6220 & 1915 & 805 & 2752 & 1155 & \textbf{131}\\
    \hline
 Cluster\_m3 \ &time & $>$1800 & $>$1800  & \cellcolor[gray]{0.8}429.5  &\  \cellcolor[gray]{0.8}60.0  & \cellcolor[gray]{0.8}23.2 & \cellcolor[gray]{0.8}\textbf{19.8}  & $>$1800\\
   & nodes & $>$4419 & $>$4197 & 3095 & 1502 & 741 & \textbf{397}  & $>$93\\
    \hline    
 \rule{0pt}{2.5ex}& & \ 1-split/\ & \ 2-split & \ 4-split & \ 8-split & \ 16-split & \ 32-split & 50-split/ \\
  & & \ big-M \ &  & &  & &  &convex hull* \\
     \hline
NN\_1 \ &time & 36.1 & \cellcolor[gray]{0.8} \textbf{29.4} & 41.8 & 57.0 & 85.7 & 145.1 & 198.5 \\
   & nodes & 24177 & 12377 & 11229 & \textbf{7415} & 11117 & 9793 & 11734 \\
     \hline
NN\_2 &time & \textbf{21.6} & 35.5 & 50.7 & 131.4 & 287.3 & 776.1 & $>$1800 \\
   & nodes & 19746 & 20157 & 14003 & 11174 & \textbf{6687} & 12685 & $>$4016 \\
    \hline
NN\_3 \ &time & \textbf{141.8} & 210.6 & 206.5 & 275.5 & 305.8 & 429.1 & 556.6\\ 
   & nodes & 116996 & 101113 & 86582 & 84455 & 69022 & 56873 & \textbf{48153}\\ 
    \hline
\end{tabular}
\footnotesize
*50-split is not the convex hull of each node for NN\_3, which has layers of 100 nodes.
\label{tab:res}
\end{table}

Note that the $P$-split formulations are in general robust towards the choice of $P$. For the clustering and P\_ball problems, all $P$-split formulations outperformed the big-M formulation both in terms of solution times and numbers of explored nodes. For the cases where the smallest $P$-split formulations timed out, Gurobi terminated with a much smaller gap compared to that of the big-M formulation. The $P$-split formulations also outperform the convex hull formulations in terms of solution time for a wide range of $P$ in all but one of the test problems.  
\footnotetext{The extended convex hull formulations for the nonlinear problems require auxiliary variables and (rotated) second order cone constraints. All $P$-split formulations  have fewer variables and constraints and only contain linear/convex-quadratic constraints.} 

For the NN problems, which have linear disjunctions, the situation is somewhat different. 
Here, while increasing $P$ still decreased the number of explored nodes, the improvements are less significant, and the trend is not completely monotonic. 
Note that bounds on the inputs to layers 2--3 are computed using interval arithmetic, resulting in overall weaker relaxations for all formulations. The weaker bounds in layers 2--3 reduce the benefits of both the $P$-split and convex hull formulations, and may favor the simpler big-M formulation.    
As the reduction in explored nodes is less drastic, smaller formulations perform the best in terms of CPU time, supporting claims that extended formulations may perform worse than expected \cite{anderson2020strong,vielma2019small}.
This may also be a consequence of Gurobi efficiently handling linear problems when it detects big-M-type constraints. 
Ignoring the big-M (1-split), the 2- and 4-splits have the lowest CPU time for all NNs, and all the split formulations solve the problems significantly faster than the convex hull formulation.  

\newpage
\section{Conclusions}
We have presented a general framework for generating intermediate relaxations in between the big-M and convex hull. The numerical results show a great potential of the intermediate relaxations, by providing a good approximation of the convex hull through a computationally simpler problem. For several of the test problems, the intermediate relaxations result in a similar number of explored nodes as the convex hull formulation while reducing the total solution time by an order of magnitude. 
\section*{Acknowledgements}
The research was funded by a Newton International Fellowship by the Royal Society (NIF\textbackslash R1\textbackslash 182194) to JK, a grant by the Swedish Cultural Foundation in Finland to JK, and by Engineering \& Physical Sciences Research Council (EPSRC) Fellowships to RM and CT (grant numbers EP/P016871/1 and EP/T001577/1). CT also acknowledges support from an Imperial College Research Fellowship.

\bibliographystyle{splncs04}
\bibliography{Ref}

\end{document}